\documentclass{mcom-l}

\usepackage{mathtools}
\usepackage{mathrsfs}
\usepackage{todonotes}
\usepackage{hyperref}
\usepackage{color,graphicx}

\usepackage[normalem]{ulem}

\newtheorem{theorem}{Theorem}[section]
\newtheorem{lemma}[theorem]{Lemma}

\newtheorem{remark}[theorem]{Remark}
\newtheorem{proposition}[theorem]{Proposition}
\newtheorem{definition}[theorem]{Definition}
\newtheorem{corollary}[theorem]{Corollary}

  \numberwithin{equation}{section}

\newcommand{\ri}{\mathrm{i}}
\newcommand{\re}{\mathrm{e}}
\newcommand{\rd}{\, \mathrm{d}}
\newcommand{\C}{\mathbb{C}}

\newcommand{\N}{\mathbb{N}}

\newcommand*{\bsc}{\boldsymbol{c}}

\newcommand*{\bsx}{\boldsymbol{x}}

\newcommand*{\bstheta}{\boldsymbol{\theta}}
\newcommand*{\bstau}{\boldsymbol{\tau}}

\newcommand{\T}{\mathbb{T}}
\newcommand{\R}{\mathbb{R}}

\newcommand{\Z}{\mathbb{Z}}

\newcommand{\calO}{\mathcal{O}}

\newcommand{\rmse}{\mathrm{rmse}}

\newcommand{\RMT}{\mathrm{RMT}}

\newcommand{\mix}{\mathrm{mix}}

\newcommand{\smon}{\mathcal{S}_+^{\rm mon}}

\usepackage{xcolor}
\definecolor{darkblue}{RGB}{0,60,180} 
\definecolor{darkgreen}{RGB}{0,130,70}
\definecolor{darkorange}{RGB}{180,60,0}
\mathtoolsset{showonlyrefs=true}

\allowdisplaybreaks

\title{M{\"o}bius-Transformed Trapezoidal Rule}

\author{Yuya Suzuki}
\address{Department of Mathematics and Systems Analysis, Aalto University School of Science, Espoo, FI-00076 Aalto, Finland}
\email{yuya.suzuki@aalto.fi}

\author{Nuutti Hyvönen}
\address{Department of Mathematics and Systems Analysis, Aalto University School of Science, Espoo, FI-00076 Aalto, Finland}
\email{nuutti.hyvonen@aalto.fi}

\author{Toni Karvonen}
\address{School of Engineering Sciences, Lappeenranta--Lahti University of Technology LUT, Yliopistonkatu 34, 53850 Lappeenranta, Finland \newline \indent Department of Mathematics and Statistics, University of Helsinki, P.O. 68 (Pietari Kalmin katu 5), 00014 University of Helsinki, Finland}
\email{toni.karvonen@helsinki.fi}

\thanks{This work was supported by the Research Council of Finland (decisions 338567, 348503, 353081, 359181 and 359183).}



\date{DATE}

\begin{document}

\maketitle
\begin{abstract}
We study numerical integration by combining the trapezoidal rule with a M{\"o}bius transformation that maps the unit circle onto the real line. We prove that the resulting transformed trapezoidal rule attains the optimal rate of convergence if the integrand function lives in a weighted Sobolev space with a weight that is only assumed to be a positive Schwartz function decaying monotonically to zero close to infinity. Our algorithm only requires the ability to evaluate the weight at the selected nodes, and it does not require sampling from a probability measure defined by the weight nor information on its derivatives. In particular, we show that the M{\"o}bius transformation, as a change of variables between the real line and the unit circle, sends a function in the weighted Sobolev space to a periodic Sobolev space with the same smoothness. Since there are various results available for integrating and approximating periodic functions, we also describe several extensions of the M{\"o}bius-transformed trapezoidal rule, including function approximation via trigonometric interpolation, integration with randomized algorithms, and multivariate integration. 
\end{abstract}


\section{Introduction}\label{sec:intro}
This paper considers numerical integration for weighted Sobolev spaces on the real line. The aim is to attain the optimal rate of convergence in terms of the smoothness of the integrand function for a wide class of weights, namely for the positive Schwartz functions that decay monotonically to zero close to infinity, dubbed simply as {\em monotonic Schwartz weights} in what follows. Here the ``optimality" of an algorithm is to be understood in the sense of the worst-case asymptotic error amongst all linear quadratures. We propose a simple algorithm that matches this optimality criterion by combining a M\"obius transformation that maps the unit circle onto the real line with the trapezoidal rule for periodic functions.  The introduced {\em M\"obius-transformed trapezoidal rule} can also be straightforwardly generalized into a randomized integration method and combined with trigonometric interpolation to introduce an algorithm for function approximation in weighted Sobolev spaces. These extensions also exhibit optimal convergence rates in terms of the smoothness of the target function. For completeness, it should be mentioned that building quadrature rules via a variable change and a subsequent application of the trapezoidal rule is definitely not a new idea~\cite{Sag1964, Schwartz1969, Stenger1973, TakahasiMori1973}, but it seems that there are no previous works proving optimal convergence of such an approach for integration in weighted Sobolev spaces on the real line.

Weighted Sobolev spaces have attracted a significant amount of attention within the numerical analysis community in the recent years. In particular, weighted integration of finitely smooth functions over unbounded domains is widely studied in the area of uncertainty quantification to tackle partial differential equations with random coefficients; see, e.g,~\cite{GKNSSS2015,KN2016,Schwab2011}.
Although such problems often require algorithms for high-dimensional integration, the present paper mainly focuses on one-dimensional numerical integration over the real line, which has the potential to lay foundations for high-dimensional counterparts, such as sparse grids \cite{BG2004} and quasi-Monte Carlo methods \cite{DKS2013}. 

Let us briefly review recent literature on numerical integration and function approximation for finitely smooth functions over unbounded domains. When the weight is a Gaussian probability density, the corresponding weighted Sobolev spaces have been considered by numerous authors; see, e.g., \cite{DILP2018,KSG2023,DK2023,GHRR2023,GKS2023} and the references therein. Freud weights, one possible generalization of Gaussian weights, have also been studied in our general context \cite{EG2024,D2024}, as have unweighted Sobolev spaces with a certain decay condition \cite{NS2023}. All aforementioned articles use the same weight for defining the Sobolev space for the target function and measuring the error in the considered approximation, which is an assumption that can be dropped:
\cite{NP2020} employs unweighted Sobolev spaces for the target function but presents weighted $L^2$ and $L^{\infty}$ error estimates, \cite{KWW2006} studies the relation between the two weights in a general framework, and univariate optimal algorithms and their convergence rates are considered in \cite{KPW2016,KPPW2020}.  
 Table~\ref{tab:summary} summarizes the settings of those aforementioned papers whose foci are close to ours.

\begin{table}
\caption{Summary of the settings in some papers mentioned in the literature review of Section~\ref{sec:intro}. 
For conciseness, we omit settings where the focus is different from ours; e.g., \cite{GHRR2023,EG2024} also contain results for infinitely smooth functions. 
Refer to Section~\ref{sec:sobolev} for the used notation. The subscript "$\mix$" here means Sobolev spaces of the dominating-mixed smoothness type; see the precise definition in each reference.}
\label{tab:summary}
\vspace{-0.7cm}
\begin{center}
\begin{tabular}{|c| c| c| c| c|} 
 \hline
  & App or Int & Source space & Target error& Note \\ [0.5ex] 
    \hline
    \hline
    \multicolumn{5}{|l|}{\kern5cm$\rho$: Gaussian}\\
    \hline 
 \cite{DILP2018} & Integration & $W_{\rho,\mix}^{\alpha,2}(\R^d)$  & $\rho$-weighted integral & \\ 
 \hline
 \cite{KSG2023}  & Integration & $W_{\rho}^{\alpha,2}(\R)$ & $\rho$-weighted integral & \\
 \hline
 \cite{DK2023}  & Both & $W_{\rho,\mix}^{\alpha,q}(\R^d)$ & $\rho$-weighted int., $L^p_{\rho}$ & $1\le p < q <\infty$\\
 & & & & and $p=q=2$ \\
 \hline
  \cite{GHRR2023}  & Both & $W_{\rho,\otimes}^{\alpha,2}(\R^d)$ & $\rho$-weighted int., $L^2_{\rho}$ &Infinite dimension \\
  \hline
  \cite{GKS2023}  & Integration & $W_{\rho}^{\alpha,2}(\R)$ & $\rho$-weighted integral &Randomized setting \\
 \hline
     \hline
    \multicolumn{5}{|l|}{\kern5cm$\rho$: Freud weight}\\
     \hline
  \cite{EG2024}  & Both & $W_{\rho}^{\alpha,2}(\R)$ & $\rho$-weighted int., $L^2_{\rho}$ & \\
    \hline
      \cite{D2024}  & Approximation & $W_{\rho,\mix}^{\alpha,q}(\R^d)$ & $L^p_{\rho}$ & $1\le p < q \le\infty$  \\
      & & & & and more
      \\
    \hline
         \hline
    \multicolumn{5}{|l|}{\kern5cm Other settings}\\
     \hline
 \cite{NS2023} & Integration &$W^{\alpha,2}_{\mix}(\R^d)$ & Unweighted integral &\\
 \hline
 \cite{NP2020} & Approximation &$W^{\alpha,2}_{\mix}(\R^d)$ & $L_{\rho}^{2}$ and $L_{\tilde{\rho}}^\infty$&\\
 \hline
 \cite{KWW2006} & Integration &$W^{1,2}_{\psi,\otimes}(\R^d)$ & $\rho$-weighted integral & $\psi\ne\rho$\\
 \hline
 Ours & Both &$W^{\alpha,q}_{\rho}(\R)$ & $\rho$-weighted int., $L^p_{\rho}$ & $1\le p <q < \infty$ \\
[1ex] 
 \hline
\end{tabular}
\end{center}
\end{table}

Concerning results in related settings, there have been studies on \emph{periodization} strategies which transform a non-periodic target function defined on a finite closed interval into a periodic one. Combined with a subsequent use of the trapezoidal rule, one can attain a faster convergence rate. We refer to \cite[Section~1]{S2006} for an overview.
However, we emphasize that for such studies on finitely smooth functions, the smoothness of the target function is often required as an input for the algorithm to attain the desired rate of convergence, whereas our method does not require this information and still achieves the optimal rate of convergence automatically.
The multivariate counterpart of periodization strategy has been studied in the context of quasi-Monte Carlo methods, more precisely, lattice rules (see, e.g., \cite{S2006,KSW2007}). For instance, \cite{GSY2019} shows that tent-transformed lattice rules can achieve second-order convergence in an appropriate function space setting, without any dimension dependence.

In addition to the optimal convergence rates (without any extra logarithmic factors) for integration and function approximation, the M\"obius-transformed trapezoidal rule also exhibits other desirable characteristics. First of all, its implementation does not require information on the smoothness of the target integrand function or the ability to sample from the probability distribution defined by the employed weight. Moreover, its capability to handle monotonic Schwartz weights, i.e., weights whose all derivatives converge to zero at infinity (only) faster than the reciprocal of any polynomial, allows to consider weights that converge to zero slower than a Gaussian density, say, only at the rate $\re^{-|x|}$ or even slower. In particular, the choice of the monotonic Schwartz weight does not affect the rate of convergence for any of the introduced algorithms. It is also worth noting that the M\"obius-transformed trapezoidal rule enables nested implementations, where function evaluations are reused when the number of quadrature points is increased, and combining it with trigonometric interpolation in function approximation allows the use of Fast Fourier Transform (FFT) to ease the computational burden.

The rest of this paper is organized as follows. In Section~\ref{sec:prelim}, we introduce and prove necessary concepts and useful lemmas related to monotonic Schwartz weights, M\"obius transformations, and Sobolev spaces. Section~\ref{sec:traped} presents our main result for numerical integration, i.e., that the M\"obius-transformed trapezoidal rule achieves the optimal rate of convergence. Section~\ref{sec:rand} extends this result to a randomized setting. In Section~\ref{sec:app}, we consider a problem of function approximation and prove the optimality of an algorithm that is based on combining the M\"obius-transformed trapezoidal rule with trigonometric interpolation. Section~\ref{sec:multidim} briefly considers a multidimensional extension of our method. Finally, Section~\ref{sec:conc} presents the concluding remarks. Induction proofs for a few technical results that are utilized in our analysis are collected in Appendix~\ref{ap:induction}.

\section{Preliminaries}\label{sec:prelim}

\subsection{Monotonic Schwartz weights}
Although the main motivation for our considerations are integrals weighted by the standard Gaussian measure, our arguments work without major modifications for a wider class of rapidly decreasing weights. To introduce our setting, consider the one-dimensional Schwartz space
\[
\mathcal{S} = \big\{ \omega \in C^{\infty}(\R) \ \big| \   \| \omega \|_{\alpha, \beta} < \infty \ \text{for all} \ \alpha, \beta \in \N_0 \big\}, 
\]
where
\[
\| \omega \|_{\alpha, \beta} = \sup_{x \in \R} | x^\alpha \omega^{(\beta)}(x) |,
\]
and $\omega^{(\beta)}$ denotes the (weak) derivative of order $\beta$. If $\| \omega \|_{\alpha, 0}$ is finite for all $\alpha \in \N_0$, we say that $\omega$ is {\em rapidly decreasing}. 

Consider the set of positive Schwartz weights on $\R$, 
\begin{equation}
\label{eq:weight_space}
\mathcal{S}_+ = \big\{ \omega \in \mathcal{S} \ | \ \omega: \R \to \R, \ \omega(x) > 0  \ \text{for all} \ x \in \R \big\},
\end{equation}
and, in particular, its subset
\begin{equation}
\label{eq:weight_space_monot}
\mathcal{S}_{+}^{\rm mon} = \big\{ \omega \in \mathcal{S}_+ \ | \  \exists K \in \R_+ \ \text{such that} \ x \, \omega'(x) \leq 0 \ \text{for all} \ x \in \R \setminus [-K, K]  \big\}
\end{equation}
consisting of functions that are monotonic on $(-\infty, -K)$ and $(K, \infty)$ for some $K > 0$. The following lemma and corollary form the basis for our treatment of these monotonic Schwartz weights.
\begin{lemma}
\label{lemma:decay}
Let $\omega \in \mathcal{S}_+^{\rm mon}$ and $\alpha, \beta \in \N_0$. For any $r > 1$,
\begin{equation}
\label{eq:log_deriv}
\left\| \frac{|\omega^{(\beta)}|^r}{\omega} \right \|_{\alpha,0} < \infty,
\end{equation}
i.e., $\frac{|\omega^{(\beta)}|^r}{\omega}$ is rapidly decreasing.
\end{lemma}

\begin{proof}
    Let $\omega \in \mathcal{S}_+^{\rm mon}$ be arbitrary and $K$ such that $\omega$ is monotonic on $(-\infty, -K)$ and $(K, \infty)$. Due to symmetry as well as the smoothness and positivity of $\omega$, it is sufficient to prove that
    \[
    \lim_{x \to \infty} x^\alpha \frac{|\omega^{(\beta)}(x)|^r}{\omega(x)} = 0
    \]
for any $r>1$ and $\alpha, \beta \in \N_0$.
    
    As the case $\beta = 0$ is trivial, assume that $\beta \in \N$. Consider a forward finite difference formula of an arbitrary order $m \in \N$,
\begin{equation}
\label{eq:fdifference0}
\bigg| \omega^{(\beta)}(x) - \frac{1}{h^\beta}\sum_{j=0}^{\beta + m-1} a_{j,\beta,m} \, \omega(x + j h) \bigg| \leq C_{\beta,m} h^{m} \| \omega \|_{0,\beta + m} = C_{\beta,m,\omega}' h^m
\end{equation}
with $x \in (K, \infty)$ and $h>0$. The existence of such  $a_{0,\beta,m}, \dots, a_{\beta + m-1, \beta,m} \in \R$ follows,~e.g.,~by applying the construction in \cite[p.~161--162]{Collatz60} to an equidistant grid and just a term of order $\beta$ in the approximated differential expression, which yields the sought-for scaling by $h^{-\beta}$ in the difference scheme. By the monotonicity of $\omega$ and the inverse triangle inequality, the estimate \eqref{eq:fdifference0} leads to
\begin{equation}
\label{eq:fdifference}
\big| \omega^{(\beta)}(x) \big| \leq 
C''_{\beta,m} \frac{\omega(x)}{h^\beta} + C'_{\beta,m,\omega} h^m \eqqcolon r(h),
\end{equation}
where $C''_{\beta,m} = \sum_{j=0}^{\beta + m-1} |a_{j,\beta,m}|$. As \eqref{eq:fdifference} holds for any $h>0$ and $r(h)$ tends to infinity when $h \to 0^+$ or $h \to \infty$, the optimal version of \eqref{eq:fdifference} is obtained at the unique zero of the derivative of $r$, i.e., at
\[
h = \left(\frac{\beta \, C''_{\beta,m} \omega(x)}{m \, C'_{\beta,m,\omega}}\right)^{1/(m+\beta)}.
\]
This finally gives
\begin{equation}
\label{eq:fdifference2}
\big| \omega^{(\beta)}(x) \big|  \leq C'''_{\beta,m,\omega} \omega(x)^{m/(m+\beta)},
\end{equation}
where the constant $C'''_{\beta,m,\omega} > 0$ is independent of $x \in (K, \infty)$.

Fix $\N \ni m > \beta/(r-1)$, which guarantees that
\[
\frac{mr}{m+\beta} > 1.
\]
Thus, by virtue of \eqref{eq:fdifference2} and because $\omega \in \mathcal{S}$,
\[
     x^\alpha \frac{|\omega^{(\beta)}(x)|^r}{\omega(x)} \leq (C'''_{\beta,m,\omega})^r x^\alpha \omega(x)^{mr/(m+\beta) -1} \longrightarrow 0
    \]
as $x \to \infty$ for any $\alpha \in \N_0$.
\end{proof}

\begin{corollary}
\label{cor:decay}
If $\omega \in \mathcal{S}_+^{\rm mon}$, then also $\omega^{s} \in \mathcal{S}_+^{\rm mon}$ for any $s>0$.
\end{corollary}

\begin{proof}
Since the function $x \mapsto x^s$ maps the open positive real axis $\R_+$ smoothly onto itself for any $s>0$, the power $\omega^{s}$ belongs to $C^\infty(\R)$ for any $\omega \in \mathcal{S}_+^{\rm mon}$. Moreover, the monotonicity of the $s$th power guarantees that $\omega^{s}$ is decreasing/increasing precisely when $\omega$ is. Hence, the assertion follows if we show that
\begin{equation}
\label{eq:sth_power}
\| \omega^s \|_{\alpha, \beta} < \infty
\end{equation}
for any $\alpha, \beta \in \N_0$.

As \eqref{eq:sth_power} immediately follows from the definition of $\mathcal{S}$ for $\beta = 0$, we may assume that $\beta \in \N$. A straightforward induction argument, presented in Appendix~\ref{ap:induction}, shows that for any $s>0$, the derivative $(\omega^s)^{(\beta)}$ is a finite linear combination of terms of the form
\begin{equation}
\label{eq:comb_terms_sth}
\omega^s \prod_{j=1}^{\gamma} \frac{\omega^{(\tau_j)}}{\omega} = \prod_{j=1}^{\gamma} \frac{\omega^{(\tau_j)}}{\omega^{1-s/\gamma}},
\end{equation}
where the positive integers $\gamma$ and $\tau_{j}$ satisfy
\begin{equation}
\gamma \leq \beta \quad \text{and} \quad \sum_{j=1}^{\gamma} \tau_j = \beta.
\end{equation}
By Lemma~\ref{lemma:decay}, 
\begin{equation}
\frac{\omega^{(\tau_j)}}{\omega^{1-s/\gamma}}
\end{equation}
is rapidly decreasing, which means that the same also applies to any term of the product form \eqref{eq:comb_terms_sth}. By the triangle inequality, $(\omega^s)^{(\beta)}$ is thus rapidly decreasing, which completes the proof.
\end{proof}

\subsection{M\"obius transformations}
\label{sec:mobius}
Our main tool for reducing integrals over $\R$ into periodic ones over $\T = \T^1$ are M\"obius transformations that map the unit circle onto the real axis of the complex plane. Here and in what follows, $\T^d$ denotes the $d$-dimensional torus, that is, $[0,2 \pi]^d$ with opposite faces identified. All M\"obius transformations (see,~e.g.,~\cite{Hille59}) with the aforementioned property can be given in the form
\begin{equation}
    \Phi_{\zeta,\vartheta}(z) = \frac{\overline{\zeta} z - {\rm e}^{{\rm i} \vartheta} \zeta}{z - {\rm e}^{{\rm i} \vartheta}}, \qquad z \in \C.
\end{equation}
Here, $\vartheta \in \R$ corresponds to a rotation prior to mapping the unit circle onto the real axis and its interior and exterior, respectively, onto the upper and lower halves of the complex plane (or vice versa). As selecting $\vartheta$ essentially only corresponds to choosing the preimage of infinity on the unit circle, it plays no essential role in our analysis and can be set to $\vartheta = 0$, meaning that $\Phi(1) = \infty$. The other free parameter $\zeta \in \C$, with a nonvanishing imaginary part, determines the image of the origin under $\Phi$. Without loss of generality, we may assume that ${\rm Re}(\zeta) = 0$ since a nonzero real part of $\zeta$ only corresponds to a horizontal translation on the image side. Hence, we set $\zeta = {\rm i}\, c$ and note that the sign of $0\not= c \in \R$ defines the half of the complex plane to which $\Phi$ maps the unit disk; from the standpoint of the restriction of $\Phi$ onto the unit circle, the choice between $\pm c$ corresponds to the orientation of the parametrization, that is, choosing whether increase in the polar angle on the unit circle leads to movement to right or left on the real axis. 

With these choices, our transformation of the unit circle onto the real line as a function of the polar angle $\theta\in (0,2\pi)$ reads
\begin{align*}
\phi_c(\theta) &= \Phi_{{\rm i} c,0}(\re^{\ri \theta})
=
- \ri c \,  \frac{\re^{\ri \theta}+1}{\re^{\ri \theta}-1}
=
-  c \,  \frac{\frac{1}{2}(\re^{\ri \theta/2}+ \re^{-\ri \theta/2})}{\frac{1}{2 \ri} (\re^{\ri \theta/2} - \re^{-\ri \theta/2})} = - c \cot\!\Big(\frac{\theta}{2}\Big).
\end{align*}
Furthermore,
\begin{equation}
\label{eq:invphi}
    \phi_c'(\theta) = \frac{c}{2 \sin^2(\theta/2)}, \quad 
    \phi_c^{-1}(x) = 2 \, {\rm arccot}\!\left(-\dfrac{x}{c}\right), \quad (\phi_c^{-1})'(x) = \frac{2 \, c}{c^2 + x^2}.
\end{equation}
 For simplicity and without severe loss of generality, we assume that $c > 0$ so that the derivatives in \eqref{eq:invphi} are positive everywhere. Unless the scaling by $c$ plays an essential role, we write $\phi(\theta)$ instead of $\phi_c(\theta)$.

\subsection{Sobolev spaces}\label{sec:sobolev}
This section provides a brief overview of the Sobolev spaces used in this paper. We denote by $L_\rho ^q (\R)$ the space of Lebesgue measurable functions $f: \R \to \C$ with the norm
\[
\|f\|_{L_\rho ^q (\R)}
\coloneqq
\bigg(\int_\R |f(x)|^q \rho(x) \rd x\bigg)^{1/q}
<
\infty, \qquad 1 \leq q < \infty.
\]
When we do not include the subscript $\rho$, i.e., write $L ^q (\R)$, we mean the unweighted space with $\rho \equiv 1$. Our target functions live in the weighted Sobolev space
\begin{align} \label{eq:gaussian-sobolev-space}
W^{\alpha,q}_{\rho}(\R) \coloneqq \bigg\{f\in L_\rho ^q (\R) \ \Big| \ \|f\|_{W^{\alpha,q}_{\rho}(\R)} \coloneqq \bigg(\sum_{\tau=0}^\alpha \int_\R |f^{(\tau)}(x)|^q \rho(x) \rd x  \bigg)^{1/q} < \infty \bigg\} 
\end{align}
for some $1 < q<\infty$ and $\alpha\in\N$, with $f^{(\tau)}$ denoting the $\tau$th weak derivative of $f$. The fact that $W^{\alpha,q}_{\rho}(\R)$ is a Banach space (or a Hilbert space for $q = 2$) for $\rho \in \smon$ is a consequence of \cite[Theorem~1.1 \& Remark~4.10]{Kufner84}. In what follows, we always consider the continuous representative of any given $f\in W^{\alpha,q}_{\rho}(\R)$, which is possible for $\alpha \in \N$ due to the Sobolev embedding theorem and the inclusion $W^{\alpha,q}_{\rho}(\R) \subset W^{1,q}_{\rm loc}(\R)$. 
In Section~\ref{sec:traped}, we specifically consider the case $q=2$ for numerical integration, even though the results also hold for any $q\in(1,\infty)$.

When suitably composed with the M\"obius transformation introduced in Section~\ref{sec:mobius}, our target functions are transformed into the periodic Sobolev space 
\begin{align}\label{eq:Lq-per-sob} 
    W^{\alpha,q}(\T) &\coloneqq \bigg\{ f\in L^q(\T)  \ \Big| \  \|f\|_{W^{\alpha,q}(\T)}^q \coloneqq \sum_{\tau=0}^{\alpha} \int_{\T} |f^{(\tau)}(x) |^q \rd x   < \infty \bigg\} 
    \end{align}
    for $1<q<\infty$ and $\alpha\in\N$. In our analysis, it is useful to employ an equivalent definition for $W^{\alpha,q}(\T)$ as a subspace of the standard Sobolev space $W^{\alpha,q}(0,2\pi)$ on $(0,2\pi)$ with compatibility conditions between the values of weak derivatives at $0$ and $2\pi$:
    \begin{align}
    \label{eq:Lq-per-sob2} 
    \nonumber
    W^{\alpha,q}_{\rm per}(0,2\pi)&=
    \bigg\{ f\in L^q(0,2\pi)  \ \Big| \  \|f\|_{W^{\alpha,q}(0,2\pi)}^q\coloneqq \sum_{\tau=0}^{\alpha} \int_0^{2\pi} |f^{(\tau)}(\theta) |^q \rd \theta  < \infty, \\[-1mm]
    & \kern4.2cm f^{(\tau)}(0)=f^{(\tau)}(2\pi) \text{ for } \tau=0,\ldots,\alpha-1  \bigg\} .
\end{align}
  The point evaluations at $0$ and $2\pi$ in \eqref{eq:Lq-per-sob2} are well-defined due to the trace theorem \cite[Theorem~7.53]{Adams1975} or the continuous embedding $W^{\alpha,q}(0,2\pi) \hookrightarrow C^{\alpha - 1, \lambda}([0,2\pi])$ for Hölder indices $0 < \lambda \leq 1 - 1/q$ \cite[Part~II in Theorem~5.4]{Adams1975}, demonstrating also that $W^{\alpha,q}_{\rm per}(0,2\pi)$ is a Banach space (or a Hilbert space for a $q = 2$) as a closed subspace of $W^{\alpha,q}(0,2\pi)$. The closure of the smooth compactly supported test functions $C^\infty_c(0, 2 \pi)$ in the topology of $W^{\alpha,q}(0,2\pi)$ is denoted by $\mathring{W}^{\alpha,q}(0,2\pi) \subset W^{\alpha,q}_{\rm per}(0,2\pi)$. Note that in \eqref{eq:Lq-per-sob} the weak derivative $f^{(\tau)}$ is defined via partial integration of periodic functions on $\T$ (see,~e.g.,~\cite[Section~5.2]{Saranen2002}), whereas in \eqref{eq:Lq-per-sob2}, $f^{(\tau)}$ denotes the standard weak derivative on $(0,2 \pi) \subset \R$, with $C^\infty_c(0, 2 \pi)$ serving as the space of test functions.

The following proposition shows that $W^{\alpha,q}(\T)$ and $W^{\alpha,q}_{\rm per}(0,2\pi)$ are, indeed, essentially the same space.

\begin{proposition}
\label{prop:periodicity}
    For $1<q<\infty$ and $\alpha\in\N$, the spaces $W^{\alpha,q}(\T)$ and $W^{\alpha,q}_{\rm per}(0,2\pi)$  can be identified via a bijective linear isometry $T: W^{\alpha,q}(\T) \to W^{\alpha,q}_{\rm per}(0,2\pi)$.
\end{proposition}

\begin{proof}
Let $\varphi: \R \ni \theta \mapsto (\cos \theta, \sin \theta)$ be the standard $2 \pi$-periodic parametrization  with respect to a polar angle $\theta$ for the embedding of $\T$ into $\R^2$. As $|\varphi'| \equiv 1$, it is straightforward to conclude that the linear mapping $T: g \mapsto g \circ \varphi$ defines a bijective isometry from $W^{\alpha,q}(\T)$ to 
\begin{align}
    \label{eq:Lq-per-sob3} 
    \nonumber
    \widetilde{W}^{\alpha,q}_{\rm per}(0,2\pi)&=
    \Big\{ f\in L^q(0,2 \pi)  \ \big| \  \|f\|_{\widetilde{W}^{\alpha,q}(0,2\pi)} < \infty \Big\},
\end{align}
where the algebraic definition of the norm $\|f\|_{\widetilde{W}^{\alpha,q}(0,2\pi)}$ is the same as that of $\|f\|_{W^{\alpha,q}(0,2\pi)}$, but the involved weak derivatives are required to satisfy the more restrictive condition 
\begin{equation}
\label{eq:weak_derivative}
\int_0^{2 \pi} f^{(\tau)}(\theta) \, \psi(\theta) \rd \theta 
= (-1)^{\tau} \int_0^{2\pi} f(\theta) \, \psi^{(\tau)}(\theta) \rd \theta, \qquad \tau = 1, \dots, \alpha, 
\end{equation}
for all $2 \pi$-periodic $\psi \in C^{\infty}(\R)$ (cf.~\cite[(5.5), (5.13) \& Exercise~5.3.2]{Saranen2002}). Hence, the assertion follows if we prove that $\widetilde{W}^{\alpha,q}_{\rm per}(0,2\pi) = W^{\alpha,q}_{\rm per}(0,2\pi)$. This boils down to showing that the weak derivatives of $f \in W^{\alpha,q}_{\rm per}(0,2\pi)$ satisfy \eqref{eq:weak_derivative} and that those of $f \in \widetilde{W}^{\alpha,q}_{\rm per}(0,2\pi)$ are compatible with the conditions on point evaluations at $0$ and $2 \pi$ in \eqref{eq:Lq-per-sob2}.

Any $f \in W^{\alpha,q}(0,2\pi)$ satisfies
\begin{align}
\label{eq:partial_integration}
\int_0^{2 \pi} f^{(\tau)}(\theta) \, \psi(\theta) \rd \theta = 
\sum_{k=1}^\tau & (-1)^{k+1} \big( f^{(\tau-k)}(2 \pi)  - f^{(\tau-k)}(0) \big) \psi^{(k-1)}(0) \nonumber \\[-1mm] & + (-1)^{\tau} \int_0^{2\pi} f(\theta) \, \psi^{(\tau)}(\theta) \rd \theta, \qquad \tau = 1, \dots, \alpha,
\end{align}
for all $2 \pi$-periodic $\psi \in C^{\infty}(\R)$, which could be proved by, e.g., approximating $f$ with smooth functions, integrating by parts and employing the continuous embedding $W^{\alpha,q}(0,2\pi) \hookrightarrow C^{\alpha - 1, \lambda}([0,2\pi])$ for $0 < \lambda \leq 1 - 1/q$  (cf.~\cite[Theorem~3.18 \& Part~II of Theorem~5.4]{Adams1975}). The claim now follows by choosing in turns $f \in W^{\alpha,q}_{\rm per}(0,2\pi)$ and $f \in \widetilde{W}^{\alpha,q}_{\rm per}(0,2\pi)$ in \eqref{eq:partial_integration} and comparing with \eqref{eq:Lq-per-sob2} and \eqref{eq:weak_derivative}.
\end{proof}

In Section~\ref{sec:multidim}, we consider a multivariate extension of our results via a componentwise M\"obius transformation. For this setting, we assume that the target function lives in a tensor product of one-dimensional weighted Sobolev spaces
\begin{equation}
\label{eq:tensor_product}
W^{\alpha,q}_{\rho,\otimes}(\R^d) \coloneqq \bigotimes_{j=1}^d W^{\alpha,q}_{\rho_j}(\R).
\end{equation}
After the componentwise M\"obius transformation, the target function is shown to be in a tensor product of periodic Sobolev spaces
\begin{equation}
\label{eq:tensor_product_torus}
W^{\alpha,q}_{\otimes}(\T^d) \coloneqq \bigotimes_{j=1}^d W^{\alpha,q}(\T).
\end{equation}
This space is known to be norm-equivalent to the Sobolev space of \emph{dominating mixed smoothness} (see, e.g., \cite{S1998})
   \begin{align}
    \label{eq:periodic-sobolev-space-d}
W^{\alpha,q}_{\rm mix}(\T^d) \coloneqq \bigg\{f\in L ^q(\T^d) \  \Big| \ \|f\|_{W^{\alpha,q} (\T^d)} \coloneqq \bigg(\sum_{|\bstau|_{\infty}\le \alpha} \int_{\T^d} |f^{(\bstau)}(\bsx)|^q  \rd \bsx  \bigg)^{1/q} \! \!< \infty \bigg\},
\end{align}
where $\bstau = (\tau_1, \dots, \tau_d) \in \N_0^d$ is a multi-index and $|\bstau|_{\infty} = \max_{j=1, \dots, d} \tau_j$.

\section{M{\"o}bius-Transformed Trapezoidal Rule} 
\label{sec:traped}
For $\rho \in \mathcal{S}_+^{\rm mon}$, denote a weighted integral of a continuous function $f: \R \to \C$ over the real line as
\begin{align*}
    I_{\rho}(f) &\coloneqq \int_\R f(x)\rho(x)\rd x
    =
    \int_0^{2\pi}f(\phi(\theta))\rho(\phi(\theta)) \phi'(\theta) \rd \theta
\end{align*}
and consider the approximation
\begin{align}\label{eq:quadrature}
    Q_{\rho,n}(f) \coloneqq \frac{2\pi}{n}\sum_{j=1}^n f(\phi(\theta_j))\rho(\phi(\theta_j)) \phi'(\theta_j) \approx I_{\rho}(f) ,
\end{align}
where $\theta_j\coloneqq 2\pi j/n$ for $j=1,\ldots,n$. 
We interpret $I_{\rho}$ and $Q_{\rho,n}$ as 
linear functionals on the $\rho$-weighted $L^2$-based Sobolev space $W^{\alpha,2}_{\rho}(\R)$.

Our main theorem is as follows:
\begin{theorem}[Upper bound on integration error]
\label{thm:main_result}
    Let $\alpha\in\N$, $\rho \in \mathcal{S}_+^{\rm mon}$ and $f\in W^{\alpha,2}_{\rho}(\R)$. Then it holds that
    \begin{equation}
    \label{eq:main_result}
        \big| I_{\rho}(f) - Q_{\rho,n}(f) \big| \leq C_{\rho, \alpha} n^{-\alpha} \| f \|_{W^{\alpha,2}_{\rho}(\R)}, 
    \end{equation}
    where $C_{\rho, \alpha} > 0$ is independent of $f$ and $n \in \N$.
\end{theorem}
The proof relies on showing that $g = ((f\rho)\circ \phi) \phi'$ belongs to the $2\pi$-periodic $L^2$-based Sobolev space $W^{\alpha,2}_{\rm per}(0,2\pi) \cong W^{\alpha,2}(\T)$
if $f\in W^{\alpha,2}_{\rho}(\R)$ for $\alpha \in \N$.
Indeed, after proving this, a classical result for the trapezoidal rule on periodic Sobolev spaces guarantees the convergence rate of order $n^{-\alpha}$; see,~e.g.,~ \cite[Proposition~7.5.6]{AH2009} and \cite[Theorem~2.4.1]{T2018_book}. 
\begin{lemma}
    \label{thm:f-rho-norm}
    Let $\alpha\in\N$, $f\in W^{\alpha,2}_{\rho}(\R)$, $\rho \in \mathcal{S}_+^{\rm mon}$ and $g = ((f\rho)\circ \phi) \phi'$. For $\tau=0,\ldots,\alpha$,
    \begin{equation}
    \label{eq:g-tau}
    \big\|g^{(\tau)} \big\|_{L^2(0,2\pi)} \leq C_{\rho, \tau} \|f\|_{W^{\alpha,2}_{\rho}(\R)},
    \end{equation}
    where the constant $C_{\rho, \tau} > 0$ is independent of $f$.
    Moreover, for $\tau=0,\ldots,\alpha-1$,
    \begin{equation}\label{eq:g-boundary}
    \lim_{\theta\to0^+} g^{(\tau)}(\theta) =
    \lim_{\theta\to 2\pi ^-} g^{(\tau)}(\theta) = 0.
    \end{equation}
    In consequence, $g\in W^{\alpha,2}_{\rm per}(0,2\pi)$.
\end{lemma}
\begin{proof}
Let us start with two auxiliary results that are proved by straightforward induction arguments in Appendix~\ref{ap:induction}. First, the weak derivative $g^{(\tau)}$, $\tau \in \N_0$, on the open interval $(0,2\pi)$ is a finite linear combination of terms of the form
\begin{equation}
\label{eq:ind1}
    \big((f^{(\tau_1)} \rho^{(\tau_2)}) \circ \phi \big) \prod_{j=1}^{\tau_1 + \tau_2+1} \phi^{(\tau_{3,j})}
\end{equation}
    where the nonnegative integers $\tau_1$, $\tau_2$, and $\tau_{3,j}$ satisfy 
    \[
    \tau_1 + \tau_2 \leq \tau \quad \text{and} \quad \sum_{j=1}^{\tau_1 + \tau_2 +1} \tau_{3,j} = \tau+1.
    \]
Secondly, 
\begin{equation}
\label{eq:ind2}
    \phi^{(\tau)}(\theta) = \frac{\psi_\tau(\theta)}{\sin^{\tau+1}(\theta/2)}, \qquad \tau \in \N_0,
\end{equation}
where $\psi_\tau \in C^\infty(\R)$ is a bounded finite linear combination of products of trigonometric functions. In particular, $g^{(\tau)}$ is continuous on $(0,2\pi)$ for $\tau=0,\ldots,\alpha-1$.

Let us first consider \eqref{eq:g-tau}. Combining \eqref{eq:ind1} and \eqref{eq:ind2}, it follows from the triangle inequality that it is, in fact, sufficient to prove \eqref{eq:g-tau} with  
   \[
    \widetilde{g}_\tau(\theta) = \frac{(f^{(\tau_1)} \rho^{(\tau_2)}) \circ \phi(\theta)}{\sin^{\eta+2}(\theta /2)}, \qquad \tau_1 + \tau_2 \leq \tau \leq \alpha, \ \ \eta = \tau + \tau_1 + \tau_2 \leq 2\tau,
    \]
replacing $g^{(\tau)}$. A direct calculation gives,
\begin{align*}
\big\|\widetilde{g}_\tau(\theta) \big\|_{L^2(0,2 \pi)}^2 &= 
    \int_0 ^{2\pi} \left|f^{(\tau_1)}(\phi(\theta)) \, \rho^{(\tau_2)}(\phi(\theta)) \, \frac{1}{\sin^{\eta+2}(\theta/2)}\right|^2 \rd \theta \\[1mm]
    & =
  \int_{\R} \left|f^{(\tau_1)}(x) \, \rho^{(\tau_2)}(x) \, \frac{1}{\sin^{\eta + 2 }(\phi^{-1}(x)/2)} \right|^2 \big|(\phi^{-1} ) ' (x)\big| \rd x
    \\[1mm]
    &\le 
    \int_{\R} \big|f^{(\tau_1)}(x) \big|^2 \rho(x) \rd x \; \sup_{x\in\R}  \left|\frac{(\phi^{-1})'(x)}{\sin^{2(\eta+2)}(\phi^{-1}(x)/2)} \frac{(\rho^{(\tau_2)}(x))^2}{\rho(x)}\right|     \\[1mm] 
    & \le C_{\tau, \tau_2} \|f\|_{W^{\alpha,2}_{\rho}(\R)}^2,
\end{align*}
where the last step follows from Lemma~\ref{lemma:decay} after expanding
\begin{align}
\label{eq:inv_sin}
\frac{(\phi^{-1})'(x)}{\sin^{2(\eta+2)}(\phi^{-1}(x)/2)} &= \frac{2 c}{c^2 + x^2} \frac{1}{\sin^{2(\eta + 2)}({\rm arccot}(-x/c))} \nonumber \\[1mm]
&= \frac{2}{c^{2\eta + 3}} \big(c^2 + x^2 \big)^{\eta+1}  
\end{align}
by virtue of \eqref{eq:invphi} and basic trigonometry. Hence, $g \in W^{\alpha, 2}(0,2 \pi)$.

To establish~\eqref{eq:g-boundary}, we demonstrate that any term of the form \eqref{eq:ind1} approaches $0$ as $\theta\to0^{\pm}$ for $\tau=0,\dots,\alpha-1$; here and in what follows, $\theta\to0^{-}$ is to be understood as the limit $\theta\to 2 \pi^-$ on the open interval $(0,2 \pi)$.
To this end, consider the following bound for any $\varepsilon\in(0,1/2)$:
\begin{align}\label{eq:g-decomp}
&\Bigg|\big((f^{(\tau_1)} \rho^{(\tau_2)}) \circ \phi (\theta) \big) \prod_{j=1}^{\tau_1 + \tau_2+1} \phi^{(\tau_{3,j})} (\theta)\Bigg|
\\[1mm]
&\le 
\left|\big(f^{(\tau_1)} \rho^{1/2+\varepsilon}\big) \circ \phi (\theta)  \right| \; \Bigg|  \frac{\rho^{(\tau_2)}}{\rho^{1/2+\varepsilon}} \circ \phi(\theta) \prod_{j=1}^{\tau_1 + \tau_2+1} \phi^{(\tau_{3,j})} (\theta)\Bigg|.\nonumber
\end{align}
Our plan is to show that the first term on the right-hand side of \eqref{eq:g-decomp} is bounded and the second one tends to zero as $\theta\to0^{\pm}$. 

Define 
\[
 h_{\tau_1} 
 \coloneqq  
f^{(\tau_1)} \rho^{1/2+\varepsilon}, \qquad \tau_1=0,\ldots,\alpha-1,
\]
which are continuous functions on $\R$. Our aim is to prove  that $h_{\tau_1} \in W^{1,2}(\R)$, so that one can conclude 
$$
\|h_{\tau_1} \|_{L^{\infty}(\R)}\le C \|h_{\tau_1} \|_{ W^{1,2}(\R)}<\infty
$$
by virtue of the Sobolev inequality~\cite[Theorem~5.4]{Adams1975}. First,
\begin{align}
\label{eq:h-estimate}
\int_\R  |h_{\tau_1} (x)|^2 \rd x 
&=
\int_\R  \big|f^{(\tau_1)}(x) \big|^2 \rho(x)^{1+2\varepsilon} \rd x \\[1mm]
&\le
\big\|f^{(\tau_1)} \big\|^2_{L^2_\rho(\R)} \, \sup_{x\in\R} \rho(x)^{2\varepsilon} \; 
<\infty
\end{align}
due to $\|f^{(\tau_1)}\|_{L^2_\rho(\R)}\le\|f\|_{W^{2,\alpha}_\rho(\R)}<\infty$.
To estimate $\|h'_{\tau_1}\|_{L^2(\R)}$, we first employ the triangle inequality:
\begin{align}
\label{eq:triangle}
\|h'_{\tau_1}\|_{L^2(\R)}& = \big\|f^{(\tau_1+1)} \rho^{1/2+\varepsilon}+(1/2+\varepsilon)f^{(\tau_1)} \rho^{\varepsilon-1/2} \rho' \big \|_{L^2(\R)} \\[1mm]
&\le
\big \|f^{(\tau_1+1)} \rho^{1/2+\varepsilon} \big\|_{L^2(\R)} +(1/2+\varepsilon)\big\|f^{(\tau_1)} \rho^{\varepsilon-1/2} \rho' \big\|_{L^2(\R)} 
.
\end{align}
Through the same line of reasoning as in \eqref{eq:h-estimate}, we obtain
\[
\big \|f^{(\tau_1+1)} \rho^{1/2+\varepsilon} \big\|_{L^2(\R)}^2
\le 
\big \|f^{(\tau_1+1)} \big\|^2_{L^2_\rho(\R)} \, \sup_{x\in\R} \rho  (x)^{2\varepsilon}
<\infty.
\]
Moreover, the second term on the right-hand side of \eqref{eq:triangle} satisfies
\begin{align*}
\big \|f^{(\tau_1)} \rho^{\varepsilon-1/2} \rho' \big \|^2_{L^2(\R)} =
&\int_\R \big|f^{(\tau_1)}(x)\big|^2 \rho(x) \frac{ \rho'(x)^2}{\rho(x)^{2-2\varepsilon}} \rd x
\\[1mm]
&\le 
\sup_{x\in\R}\left(\frac{ \rho'(x)^2}{\rho(x)^{2-2\varepsilon}}\right)\; \big \|f^{(\tau_1)} \big \|^2_{L^2_\rho(\R)}
<\infty,
\end{align*}
where the last inequality follows from Lemma~\ref{lemma:decay}. Hence, 
\begin{equation}
\sup_{\theta \in (0,2\pi)} \Big| \big(f^{(\tau_1)} \rho^{1/2+\varepsilon}\big) \circ \phi(\theta) \Big|  = \|h_{\tau_1}\|_{L^{\infty}(\R)} <\infty, 
\end{equation}
which demonstrates the boundedness of the first term on the right-hand side of~\eqref{eq:g-decomp}.

 To conclude the proof of \eqref{eq:g-boundary}, consider the second term on the right-hand side of \eqref{eq:g-decomp}. By Lemma~\ref{lemma:decay}, we know that 
 $$
 \lim_{x \to \pm \infty} x^\gamma \frac{\rho^{(\tau_2)}(x)}{\rho(x)^{1/2+\varepsilon}} = 0
 $$ 
 for any $\gamma \in \N_0$. Combining this observation with \eqref{eq:ind2} yields
\[
\lim_{\theta\to0^\pm} \Bigg| \frac{\rho^{(\tau_2)}}{\rho^{1/2+\varepsilon}} \circ \phi(\theta) \prod_{j=1}^{\tau_1 + \tau_2+1} \phi^{(\tau_{3,j})} (\theta)\Bigg| =0,
\]
which establishes \eqref{eq:g-boundary}; see also \eqref{eq:inv_sin}.

The final conclusion that $g\in W^{\alpha,2}_{\rm per}(0,2\pi)$ follows from the definition~\eqref{eq:Lq-per-sob2}. 
\end{proof}

We demonstrate our method numerically for the integrand function $f(x)=|x|^p$, with $p\in\{1,3,5\}$, which is $p$ times weakly differentiable and in $W_\rho^{p,2}(\R)$, but not in $W_\rho^{p+1,2}(\R)$ for $\rho \in \smon$. We use the value $c=1$ for the free parameter in the M\"obius transformation but note that its choice does not seem to have an effect on the observed asymptotic convergence rates. Two rapidly decreasing weights are considered: the standard Gaussian $\rho(x)=\re^{-x^2/2}/\sqrt{2\pi}$ and logistic $\rho(x)=\re^{-x}/(1+\re^{-x})^2$ densities. The tests are performed on Matlab 2022b with double precision arithmetic. 

Figure~\ref{fig:int-test}  compares the error convergence for the M\"obius-transformed trapezoidal rule for the Gaussian weight with two other methods, i.e., the Gauss--Hermite quadrature and the trapezoidal rule with a cut-off from \cite{KSG2023}.
Regarding the logistic distribution, Figure~\ref{fig:int-test2} shows a comparison between the M\"obius-transformed trapezoidal rule and the Gauss--Logistic quadrature for which we use the Matlab software by Walter Gautschi \cite{G2020}. Interestingly, Gaussian quadrature shows much slower rate of convergence in the latter case. This slow convergence of the Gauss--Logistic quadrature is not due to the mere double precision in our computations: if the higher precision supported by the implementation of the Gauss--Logistic quadrature is used, the decay rate of the error remains the same.
\begin{figure}
\centering
 \includegraphics{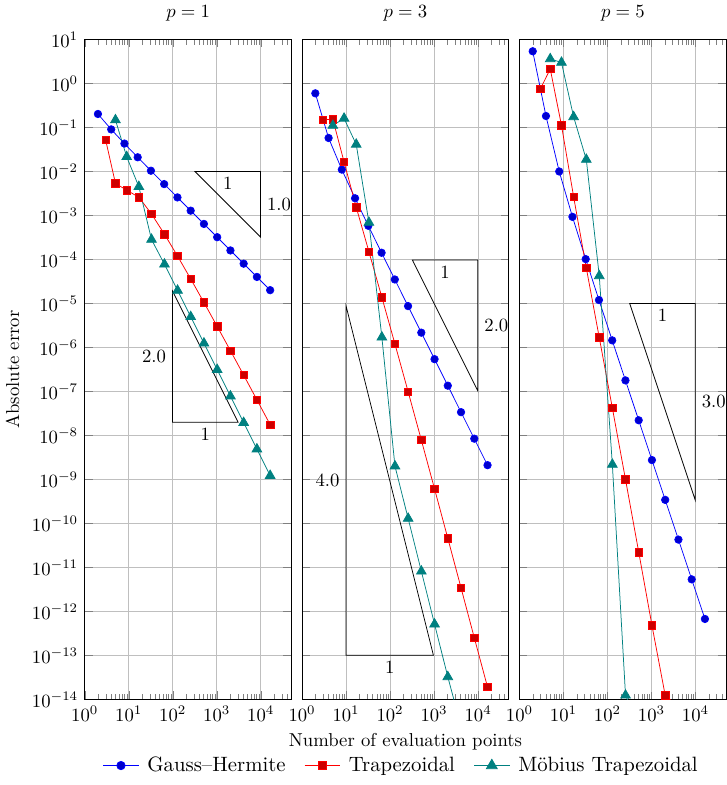}
\caption{Absolute integration error for the Gaussian weight and $f(x)=|x|^p$, which corresponds to $I_\rho(f) = (2^p/\pi)^{1/2}\Gamma((p+1)/2)$ where $\Gamma$ is the Gamma function. The blue line shows the error for the Gauss--Hermite quadrature and the red line for the trapezoidal rule with a cut-off from \cite{KSG2023}. The M\"{o}bius-transformed trapezoidal rule (green) achieves the fastest convergence of the error.}
\label{fig:int-test}
\end{figure}

\begin{figure}
\centering
 \includegraphics{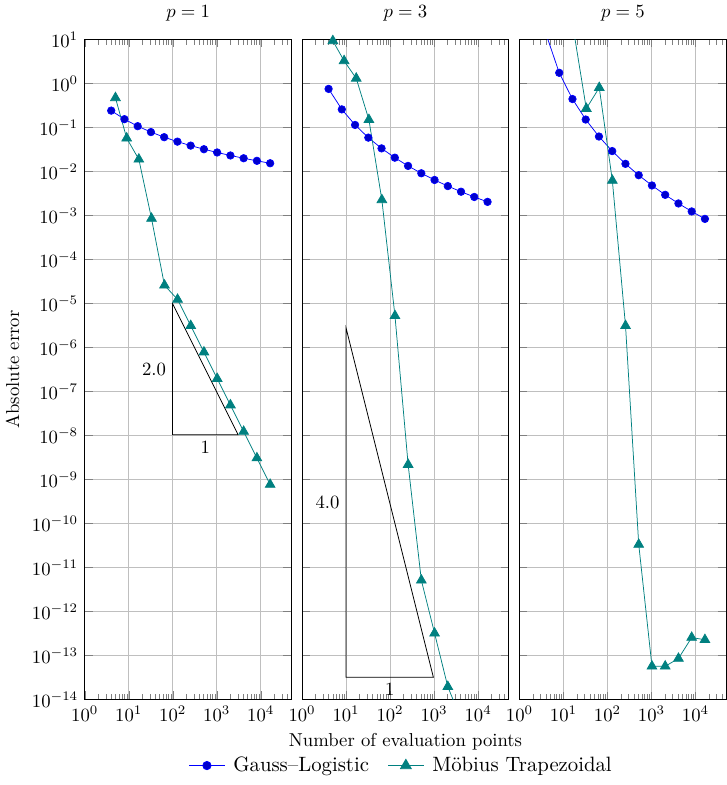}
\caption{Absolute integration error for the logistic weight and $f(x)=|x|^p$, which corresponds to $I_\rho(f) = -2 \,  p! \, {\rm Li}_p(-1)$ where ${\rm Li}$ denotes the polylogarithm \cite[Corollary~4.2]{Jodra2014}. The blue line shows the error for the Gauss--Logistic quadrature from \cite{G2020}. The M\"{o}bius-transformed trapezoidal rule (green) exhibits much faster convergence.}
\label{fig:int-test2}
\end{figure}

\begin{remark}[Implementation]
From a standpoint of an implementation on a computer, it should be pointed out that certain function evaluations in \eqref{eq:quadrature} at the M\"obius-transformed equidistant points on the unit circle may cause numerical errors due to the unbounded values of $\phi$ and $\phi'$ at $0$ and $2\pi$. This problem can be easily circumvented in two alternative ways: (i) by \eqref{eq:g-boundary} we know that $g(0)=g(2\pi)=0$, which means that one can set the integrand function to zero at $0$ and $2\pi$ without actually evaluating $\phi$ or $\phi'$ at those points; or (ii) one can add a shift $\Delta\in(0,2\pi/n)$ to all equidistant quadrature points on the unit circle, which leads to evaluating the integrand in \eqref{eq:quadrature} at $\theta_j'=2\pi j/n + \Delta$ for $j=1,\ldots,n$.

    Moreover, since our M\"obius-transformed quadrature is based on the trapezoidal rule on the unit circle, it is straightforward to implement a nested/embedded version of the algorithm such that one can reuse previous function evaluations when increasing the number of quadrature points $n$ in \eqref{eq:quadrature}.
\end{remark}

To support our claim that the M\"obius-transformed trapezoidal rule achieves the optimal rate of worst-case error amongst any linear quadrature for the considered integration problem with a monotonic Schwartz weight $ \rho \in \smon$, we still need to deduce a matching general lower bound for this class of algorithms. We follow the argument in \cite[Theorem~2.3]{DK2023}, where the authors consider the standard Gaussian weight. 
\begin{proposition}[General lower bound]
    Let $\rho\in\smon$, $1<q<\infty$, and $\alpha\in\N$. Then, for any linear quadrature of the form 
    \[A_{\rho,n}(f)\coloneqq \sum_{j=1}^n w_j f(x_j), \qquad w_j\in\R,\;\;x_j\in\R, \]
    we have
    \[
    \big \| I_\rho - A_{\rho,n} \big\|_{\mathscr{L}(W_\rho^{\alpha,q}(\R), \C)} =  \sup_{0\ne f\in W_\rho^{\alpha,q}(\R)} \frac{\left| I_\rho (f)-A_{\rho,n}(f)\right|}{\|f\|_{W_\rho^{\alpha,q}(\R)}}\ge C \frac{1}{n^{\alpha}},
    \]
    where the constant $C$ is independent of $n$.
\end{proposition}
\begin{proof}
    The space $\mathring{W}^{\alpha,q}(0,1)$ can be interpreted as a subspace of $W^{\alpha,q}(\R)$ via zero continuation of its elements onto $\R\setminus (0,1)$ \cite[Lemma~3.22]{Adams1975}, and thus also $\mathring{W}^{\alpha,q}(0,1) \subset W^{\alpha,p}_{\rho}(\R)$. Since
    \[
    {\max_{x\in[0,1]}\rho(x)^{1/q}} \, \Big( \int_0 ^1  \big| g(x) \big|^q \rd x \Big)^{1/q} \ge  \Big( \int_\R  \big| g(x) \big|^q \rho(x) \rd x \Big)^{1/q}
    \]
    for all $g\in L^{q}(0,1)$, including the first $\alpha$ partial derivatives of $f\in \mathring{W}_\rho^{\alpha,q}(0,1)$, we have
    \begin{align}
    \label{eq:lowerbound}
    &\sup_{0\ne f\in W_\rho^{\alpha,q}(\R)} \frac{\left| I_\rho (f)-A_{\rho,n} (f) \right|}{\|f\|_{W_\rho^{\alpha,q}(\R)}}
 \nonumber   \\& \kern2cm \ge
    \frac{1}{\max_{x\in[0,1]}\rho(x)^{1/q}} \, 
\sup_{0\ne f\in \mathring{W}^{\alpha,q}(0,1)} \frac{\left| I_\rho (f) -A_{\rho,n} (f) \right|}{\|f\|_{W^{\alpha,q}(0,1)}} \nonumber \\[1mm]
& \kern2cm =   \frac{1}{\max_{x\in[0,1]}\rho(x)^{1/q}}  \, \sup_{0\ne f\in \mathring{W}^{\alpha,q}(0,1)} \frac{\left| I(f) -A_{\rho,n} (f/\rho) \right|}{\|f/\rho \|_{W^{\alpha,q}(0,1)}},  
    \end{align}
    where $I(f)$ denotes the unweighted integral of $f$ over $\R$ and the last step is a consequence of the mapping $R: f \mapsto f/\rho$ being a linear homeomorphism on $\mathring{W}^{\alpha,q}(0,1)$ due to the positivity and smoothness of $\rho$ on $[0,1]$. In particular,
    \begin{align}
    &\sup_{0\ne f\in \mathring{W}^{\alpha,q}(0,1)}  \frac{\left| I(f) -A_{\rho,n} (f/\rho) \right|}{\|f/\rho \|_{W^{\alpha,q}(0,1)}} \\
    & \qquad \geq \frac{1}{\|R \|_{\mathscr{L}(\mathring{W}^{\alpha,q}(0,1))}} \, \sup_{0\ne f\in \mathring{W}^{\alpha,q}(0,1)} \frac{\left| I(f) -A_{\rho,n} (f/\rho) \right|}{\|f \|_{W^{\alpha,q}(0,1)}} \geq C_{\rho,\alpha} \frac{1}{n^{\alpha}},
    \end{align}
    where the last inequality corresponds to a lower bound for the accuracy of linear quadrature rules for unweighted integrals of functions in $\mathring{W}^{\alpha,q}(0,1)$ over $(0,1)$ \cite{T1990}. Combined with \eqref{eq:lowerbound}, this proves the claim.
\end{proof}

\section{Randomized trapezoidal rule}\label{sec:rand}
The considered integration problem for the important special case of a Gaussian weight is tackled by randomized algorithms in \cite{GKS2023}. In particular, the best attainable worst-case root-mean-squared error (RMSE), amongst any possibly nonlinear or adaptive algorithm, is proved to be of order $n^{-\alpha-1/2}$ \cite[Theorem~2.1]{GKS2023}. Using our M\"{o}bius-transformed trapezoidal rule, with a suitable randomization, one can attain this optimal rate, without a logarithmic multiplicative factor as in \cite[Theorem~3.3]{GKS2023} for a truncated randomized trapezoidal rule. 

In the following, we call $A$ a \emph{randomized algorithm}, which is a pair of a probability space $(\Omega, \Sigma, \mu)$  and a family of mappings $(A^{\omega})_{\omega\in\Omega}$, when (i) each fixed $\omega\in\Omega$ defines a deterministic algorithm $A^{\omega}$ and (ii) the number of nodes used for each fixed integrand $f$ is measurable with respect to $\omega$.
We define the worst-case RMSE for  a randomized algorithm $A$ on $W^{\alpha,2}_{\rho}(\R)$ as
   \[
   e^{{\rmse}}\big(A,W^{\alpha,2}_{\rho}(\R)\big)
   \coloneqq
   \sup_{0\ne f\in W^{\alpha,2}_{\rho}(\R)} \frac{\Big(\int_{\Omega} \big|I_{\rho}(f)-A^{\omega} (f) \big|^2 \rd \mu(\omega) \Big)^{1/2}}{\|f\|_{W^{\alpha,2}_{\rho}(\R)}}.
   \]

\begin{definition}[Randomized M\"{o}bius-transformed trapezoidal rule]
Let $M$ be an integer-valued random variable that is distributed uniformly over $\{\lfloor\frac{n}{2}\rfloor,\ldots,n\}$, and let $\delta$ be a uniformly distributed random variable on $[0,1]$. Assume that $M$ and $\delta$ are mutually independent. For a continuous function $f:\R\to\C$, we define a weighted randomized M\"{o}bius-transformed trapezoidal rule $A_{n,\RMT}=(A_{n,\RMT}^{M ,\delta})_{{M ,\delta}}$ by 
\begin{equation}\label{def:rand-MRT}
    A_{n,\RMT}^{M ,\delta} (f)
    \coloneqq
    \frac{2\pi}{M } \sum_{j=0}^{M -1}
    f(\phi(\theta_j))\, \rho(\phi(\theta_j))\, \phi'(\theta_j)
    ,
\end{equation}
    where the integration nodes are defined as $\theta_j\coloneqq 2\pi(j+\delta)/M $.
\end{definition}

The convergence rate $n^{-\alpha-1/2}$ of the RMSE for the randomized quadrature rule \eqref{def:rand-MRT} is a direct consequence of Lemma~\ref{thm:f-rho-norm} combined with the classical result by Bakhvalov, \cite{B1961} or \cite[Theorem~11]{KKNU2019} where multidimensional settings are considered. Since our setting is one-dimensional, we refer to \cite[p. 1670]{GKS2023} for a better bound on the RMSE for $2T$-periodic functions. For completeness, we formally state this result as follows:

\begin{theorem}[Upper bound on randomized integration]
   For $\rho \in \mathcal{S}_+^{\rm mon}$ and $\alpha \in \N$, the randomized M\"{o}bius-transformed trapezoidal rule \eqref{def:rand-MRT} satisfies 
   \[
   e^{{\rmse}}\big(A_{n,\RMT},W^{\alpha,2}_{\rho}(\R)\big)
   \le
   C n^{-\alpha-1/2},
   \]
   where $C>0$ is independent of $n$.
\end{theorem}

\section{$L^p_{\rho}$ approximation}\label{sec:app}
In this section, we consider approximating functions in the $L^p_{\rho}(\R)$-norm. More precisely, we aim to construct an algorithm $A_n: W^{\alpha,q}_{\rho}(\R) \to L^p_{\rho}(\R)$ using $n$ function evaluations, so that the the worst-case error 
\begin{align}\label{eq:wce-approx}
\| I - A_n \|_{\mathscr{L}(W^{\alpha,q}_{\rho}(\R),L^p_{\rho}(\R))} 
= \sup_{0\ne f\in W^{\alpha,q}_{\rho}(\R)} \frac{\|f-A_n f \|_{L^p_{\rho}(\R)}}{\|f\|_{W^{\alpha,q}_{\rho}(\R)}}
\end{align}
is as small as possible for $1 \leq p < q < \infty$.

To begin with, let us expand the $L^p_{\rho}(\R)$-error as follows:
\begin{align*}
\|f-A_n f\|^p_{L^p_{\rho}(\R)}&=\int_\R \left|f(x)-(A_n f)(x)\right|^p\rho(x)\rd x\\
&=
\int_0^{2\pi} \big|f(\phi(\theta)) \big(\rho(\phi(\theta)) \phi'(\theta) \big)^{1/p} -(B_n f)(\theta)\big|^p  \! \rd \theta
,
\end{align*}
where the transformed approximation operator $B_n$ on the unit circle is defined by
\[
(B_n f)(\theta) = (A_n f)(\phi(\theta)) \, \big(\rho(\phi(\theta)) \phi'(\theta) \big)^{1/p}. 
\]
Hence, deviating slightly from the integration problem of Section~\ref{sec:traped}, our aim is to construct an approximation for $g_p \coloneqq ((f \rho^{1/p}) \circ \phi) (\phi')^{1/p}$ on the torus $\T$.
We propose the following algorithm, which is nothing but trigonometric interpolation of $g_p$ using equidistant points on the unit circle:
\begin{definition}[Trigonometric interpolation with M\"obius transformation]
\label{def:trig_interp}
Let $n\in\N$ and $g_p \coloneqq ((f \rho^{1/p}) \circ \phi) (\phi')^{1/p}$. We define the algorithm $B_n$ via trigonometric interpolation of $g_p$:
\begin{align}\label{eq:approx-alg} 
(B_n f)(\theta) \coloneqq \sum_{k=\lfloor -(n-1)/2 \rfloor}^{\lfloor (n-1)/2 \rfloor} \widehat{g_p}(k) \re^{\ri k \theta}, \qquad  \widehat{g_p}(k) \coloneqq \frac{1}{2\pi n}\sum_{j=0}^{n-1} g_p (\theta_j) \re^{-\ri k \theta_j}
,
\end{align}
where $\theta_j = 2\pi j / n$.
The resulting algorithm over the real line is then given by
\[
(A_n f) (x)=(B_n f) (\phi^{-1}(x))\, \big(\rho(x)\phi'(\phi^{-1}(x) \big)^{-1/p},
\]
where
\[
\phi'(\phi^{-1}(x)) = \frac{1}{(\phi^{-1})'(x)} = \frac{1}{2c} \big(x^2 + c^2\big)
\]
by \eqref{eq:invphi}.
\end{definition}
We claim that this algorithm achieves the optimal rate of convergence; this result is presented in two parts as Theorem~\ref{thm:main_result_p} and Proposition~\ref{prop:p_lower_bound}.
\begin{theorem}[Upper bound on $L^p_\rho$ approximation]\label{thm:Lp-approx}
\label{thm:main_result_p}
    For $\rho \in \mathcal{S}_+^{\rm mon}$ and $1\le p <q<\infty$, it holds that
    \begin{equation}
    \label{eq:main_result_p}
        \| I - A_n \|_{\mathscr{L}(W^{\alpha,q}_{\rho}(\R), L^p_{\rho}(\R))} \leq C_{\rho, \alpha,p,q} n^{-\alpha}, \quad \alpha \in \N, 
    \end{equation}
    where $C_{\rho, \alpha,p,q} > 0$ is independent of $n \in \N$.
\end{theorem}

To deduce the convergence rate of Theorem~\ref{thm:Lp-approx}, we generalize Lemma~\ref{thm:f-rho-norm} to show that the transformed target function $g_p$ belongs to the periodic $L^q$-based Sobolev space $W^{\alpha,q}_{\rm per}(0, 2 \pi) \cong W^{\alpha,q}(\T)$ for $q > p$. The proof of Theorem~\ref{thm:Lp-approx} then follows from a trigonometric interpolation result by Temlyakov \cite[Theorem~2.7]{T2018_book}. 

\begin{lemma}\label{lemma:f-rho-Lq}
        Let $\alpha\in\N$, $1\le p<q<\infty$, $f\in W^{\alpha,q}_{\rho}(\R)$, $\rho \in \mathcal{S}_+^{\rm mon}$ and $g_p = ((f\,\rho^{1/p})\circ \phi) (\phi')^{1/p}$.
        For $\tau=0,\ldots,\alpha$, 
    \begin{equation}
    \label{eq:g_p-tau}
    \big\|g_p^{(\tau)} \big\|_{L^q(0, 2 \pi)} \leq C_{\rho, \tau, p, q} \|f\|_{W^{\alpha,q}_{\rho}(\R)},
    \end{equation}
    where the constant $C_{\rho, \tau, q, p} > 0$ is independent of $f$.
    Moreover, for $\tau=0,\ldots,\alpha-1$,
    \begin{equation}\label{eq:g_p-boundary}
    \lim_{\theta\to0^+} g_p^{(\tau)}(\theta) =
    \lim_{\theta\to 2\pi^-} g_p^{(\tau)}(\theta) = 0.
    \end{equation}   
    In consequence, $g_p\in W^{\alpha,q }_{\rm per}(0,2\pi)$.
\end{lemma}
\begin{proof}
The proof follows the general structure of that for Lemma~\ref{thm:f-rho-norm}. To begin with, note that $\rho^{1/p} \in \smon$ due to Corollary~\ref{cor:decay}. 

A straightforward induction argument, presented in Appendix~\ref{ap:induction}, demonstrates that the weak derivative $g_p^{(\tau)}$, $\tau \in \N$, on the open interval $(0, 2 \pi)$ is a finite linear combination of terms of the form
\begin{equation}
\label{eq:ind1p}
    \big((f^{(\tau_1)} (\rho^{1/p})^{(\tau_2)}) \circ \phi \big) (\phi')^{1/p - \tau_3} \prod_{j=1}^{\tau_1+\tau_2 + \tau_3} \phi^{(\tau_{4,j})},
\end{equation}
    where the nonnegative integers $\tau_1$, $\tau_2$, $\tau_3$ and $\tau_{4,j}$ satisfy 
    \[
    \tau_1 + \tau_2 + \tau_3 \leq \tau \quad \text{and} \quad \sum_{j=1}^{\tau_1+\tau_2 + \tau_3} \tau_{4,j} = \tau + \tau_3.
    \]
    Recall that the structure of $\phi^{(\tau)}$ is as indicated in \eqref{eq:ind2}; see also \eqref{eq:invphi}.

Let us first tackle \eqref{eq:g_p-tau}. Because of \eqref{eq:ind1p}, \eqref{eq:ind2}, \eqref{eq:invphi} and the triangle inequality, it is sufficient to prove \eqref{eq:g_p-tau} with 
\[
\widetilde{g}_{\tau, p}(\theta) 
\coloneqq
\frac{\big(f^{(\tau_1)} (\rho^{1/p})^{(\tau_2)}\big) \circ \phi (\theta) }{\sin^{\eta+2/p}(\theta/2)}, \qquad \tau_1 + \tau_2 \leq \tau \leq \alpha, \ \ \eta = \tau + \tau_1 + \tau_2 \leq 2\tau,
\]
replacing $g_p^{(\tau)}$.
We have
\begin{align*}
    \big\|\widetilde{g}_{\tau, p} \big\|_{L^q(0,2 \pi)}^q &= 
    \int_0 ^{2\pi} \bigg|f^{(\tau_1)}(\phi(\theta)) \, (\rho^{1/p})^{(\tau_2)}(\phi(\theta)) \, \frac{1}{\sin^{\eta+2/p}(\theta/2)}\bigg|^q \rd \theta \\[1mm]
    & =
  \int_{\R} \bigg|f^{(\tau_1)}(x) \, (\rho^{1/p})^{(\tau_2)}(x) \, \frac{1}{\sin^{\eta + 2/p }(\phi^{-1}(x)/2)} \bigg|^q \big|(\phi^{-1} ) ' (x)\big| \rd x
    \\[1mm]
    &\le 
    \int_{\R} \big|f^{(\tau_1)}(x) \big|^q \rho(x) \rd x \\ 
    & \qquad \qquad  \times \ \ \sup_{x\in\R}   \frac{\big|(\rho^{1/p})^{(\tau_2)})(x)\big|^q}{(\rho^{1/p})(x)^p} \frac{|(\phi^{-1} )' (x)|}{\big| \sin^{q \eta + 2 q/p}(\phi^{-1}(x)/2)\big|}  
    \\[2mm]
    &\le C_{p,q,\tau_2} \|f\|^q _{W^{\alpha,q}_{\rho}(\R)}.
\end{align*}
The last inequality follows from Lemma~\ref{lemma:decay} with $\omega=\rho^{1/p}$ and $r=q/p>1$ (cf.~Corollary~\ref{cor:decay}) since
\begin{align}
\label{eq:inv_sinp}
\frac{(\phi^{-1} )' (x)}{ \sin^{q \eta + 2 q/p}(\phi^{-1}(x)/2))} &= \frac{2 c}{c^2 + x^2} \frac{1}{\sin^{q \eta + 2 q/p}({\rm arccot}(-x/c))} \\[1mm] 
&= \frac{2}{c^{q \eta + 2 q/p -1}} \big(c^2 + x^2 \big)^{q \eta/2 + q/p -1}
\end{align}
by elementary trigonometry and \eqref{eq:invphi}.

To prove \eqref{eq:g_p-tau}, we start by bounding a term of the form \eqref{eq:ind1p} for any $\varepsilon\in(0,(q-p)/pq)$ as follows:
    \begin{align}
    \label{eq:p_convergence}
& \qquad \Bigg|\big((f^{(\tau_1)} (\rho^{1/p})^{(\tau_2)}) \circ \phi (\theta) \big) \phi'(\theta)^{1/p - \tau_3} \prod_{j=1}^{\tau_1+\tau_2+\tau_3} \phi^{(\tau_{4,j})} (\theta)\Bigg|
\\[1mm]
&\le 
 \Big|\big(f^{(\tau_1)} \rho^{1/q+\varepsilon}\big) \circ \phi (\theta)  \Big| \; \Bigg|  \bigg( \frac{(\rho^{1/p})^{(\tau_2)}}{\rho^{1/q+\varepsilon}} \circ \phi(\theta) \bigg) \phi'(\theta)^{1/p - \tau_3} \prod_{j=1}^{\tau_1 + \tau_2+\tau_3}\phi^{(\tau_{4,j})} (\theta)\Bigg|.
\end{align}
Since $1/q + \varepsilon <1/p$, it is a consequence of Lemma~\ref{lemma:decay}, Corollary~\ref{cor:decay}, \eqref{eq:invphi} and \eqref{eq:ind2} that 
\begin{equation}
\label{eq:p-limits}
\lim_{\theta\to\pm0}\Bigg|  \bigg( \frac{(\rho^{1/p})^{(\tau_2)}}{\rho^{1/q+\varepsilon}} \circ \phi(\theta) \bigg) \phi'(\theta)^{1/p - \tau_3} \prod_{j=1}^{\tau_1 + \tau_2+\tau_3}\phi^{(\tau_{4,j})} (\theta) \Bigg|=0;
\end{equation}
see also \eqref{eq:inv_sinp}.

Consider then the first term on the right-hand side of \eqref{eq:p_convergence}, or more precisely, the continuous function $h_{q,\tau_1}\coloneqq f^{(\tau_1)} \rho^{1/q+\varepsilon}$, $\tau_1 = 0, \dots, \tau \leq \alpha-1$, with the weak derivative 
\begin{equation}
\label{eq:hp_deriv}
h'_{q,\tau_1}=f^{(\tau_1+1)} \rho^{1/q+\varepsilon} +(1/q+\varepsilon) f^{(\tau_1)} \rho^{1/q+\varepsilon-1} \rho'.
\end{equation}
By showing that both $h_{q,\tau_1}$ and $h'_{q,\tau_1}$ belong to $L^q(\R)$, the (essential) boundedness of $h_{q,\tau_1}$ follows from the Sobolev embedding $W^{1,q}(\R) \hookrightarrow L^\infty(\R)$. 
We have
\[
\|h_{q,\tau_1}\|^q_{L^q(\R)}
=
\int_{\R} \big|f^{(\tau_1)}(x)\big|^q \rho^{1+q\varepsilon}(x)\rd x 
\le
\big\|f^{(\tau_1)} \big\|^q_{L^q_\rho(\R)} \sup_{x \in \R}\rho^{q\varepsilon}(x)
<
\infty.
\]
Moreover, the two terms composing $h'_{p,\tau_1}$ in \eqref{eq:hp_deriv} satisfy
\[
\big\|f^{(\tau_1+1)} \rho^{1/q+\varepsilon} \big\|^q_{L^q(\R)} 
\le
\big\|f^{(\tau_1+1)} \big\|^q_{L^q_\rho(\R)} \sup_{x \in \R}\rho^{q\varepsilon}(x)
<
\infty,
\]
and
\begin{align*}
    \big \|f^{(\tau_1)} \rho^{1/q+\varepsilon-1} \rho' \big\|^q_{L^q(\R)} 
    &=
    \int_\R \big|f^{(\tau_1)} (x)\big|^q \rho(x)\,
    \frac{ |\rho'(x)|^q }{\rho(x)^{q-q\varepsilon}} \rd x
    \\[1mm]
    &\le
    \big\|f^{(\tau_1)} \big\|^q_{L^q_\rho(\R)} \,
    \sup_{x \in \R}  \frac{ |\rho'(x)|^q }{\rho(x)^{q-q\varepsilon}}
<
\infty,
\end{align*}
where the last inequality is a consequence of Lemma~\ref{lemma:decay}. Combining these estimates with \eqref{eq:p-limits} proves \eqref{eq:g_p-boundary}.

The final conclusion that $g_p \in W^{\alpha,q}_{\rm per}(0,2\pi)$ follows from the definition~\eqref{eq:Lq-per-sob2}. 
\end{proof}

In order to prove the claimed optimality of the M\"obius-transformed trigonometric interpolation, we also give a lower bound for all linear approximation algorithms.

\begin{proposition}[General lower bound]
\label{prop:p_lower_bound}
Let $\rho \in \smon$, $1 \leq p < q < \infty$, and $\alpha \in \N$. Then, for any linear operator $A_n: W^{\alpha,q}_{\rho}(\R) \to L^p_{\rho}(\R)$ with ${\rm rank}(A_n) \leq n$,
\begin{equation}
 \| I - A_n \|_{\mathscr{L}(W^{\alpha,q}_{\rho}(\R), L^p_{\rho}(\R))} \geq C \frac{1}{n^\alpha},
\end{equation}
where the constant $C$ is independent of $n$ and $A_n$.
\end{proposition}

\begin{proof}
    We follow the line of reasoning in \cite[Proof of Theorem~3.3]{DK2023}. To this end, let $f \in W^{\alpha,q}_{\rm per}(0,1)$ and denote its $1$-periodic extension onto the real line by the same symbol. 
    An argument similar to that in the proof of Proposition~\ref{prop:periodicity} demonstrates that the extension $f$ belongs to $W^{\alpha,q}_{\rm loc}(\R)$. For any $N \in \N$,
    \begin{align}
    \label{eq:est_per}
        \| f \|_{W^{\alpha,q}_{\rho}(\R)}^q &= \sum_{\tau = 0}^\alpha \int_{\R} \big| f^{(\tau)} (x) |^q \rho(x) \rd x \nonumber \\
        &= \sum_{\tau = 0}^\alpha \sum_{k \in \Z} \int_0^{1} \big| f^{(\tau)} (x+k) |^q \rho(x+k) \rd x \nonumber \\
        &\leq  \sum_{\tau = 0}^\alpha \int_0^{1} \big| f^{(\tau)} (x) |^q  \rd x \, \sum_{k \in \Z} \sup_{x \in [0,1]} \rho(x + k) \nonumber \\
        & \leq C_{\rho, N} \|  f \|_{W^{\alpha,q}(0,1)}^q \sum_{k \in \Z}
        \sup_{x \in [0,1]} \frac{1}{(1 + (x + k)^{2})^N} \nonumber \\
        & \leq C_{\rho} \|  f \|_{W^{\alpha,q}(0,1)}^q
    \end{align}
    since $\rho$ is rapidly decreasing. In particular, $f \in W^{\alpha,q}_{\rho}(\R)$.

    Since the above construction applies to any $f \in W^{\alpha,q}_{\rm per}(0,1)$ with the same constant $C_{\rho}$, we have 
\begin{align}
& \| I - A_n \|_{\mathscr{L}(W^{\alpha,q}_{\rho}(\R), L^p_{\rho}(\R))} = \sup_{0 \not= f \in W^{\alpha,q}_{\rho}(\R)} \frac{\|(I - A_n) f \|_{L^{p}_{\rho}(\R)}}{\| f \|_{W^{\alpha, q}_{\rho}(\R)}} \\[1mm]
& \qquad \qquad \qquad \geq \frac{\min_{x \in [0,1]} \rho(x)^{1/p}}{C_\rho^{1/q}} \sup_{0 \not= f \in W^{\alpha}_{\rm per}(0,1)} \frac{\|(I - A_n) f \|_{L^{p}(0,1)}}{\| f \|_{W^{\alpha, q}(0, 1)}}.
       \end{align}  
       The claim then follows from a lower bound for the approximation of periodic functions \cite[Theorem~2.1.1]{T2018_book}, presented originally in \cite{M1972}.
\end{proof}

\begin{remark}[Fast Fourier Transform]
    If one resorts to the Fast Fourier Transform (FFT) in \eqref{eq:approx-alg}, the computational cost and memory usage by the algorithm of Definition~\ref{def:trig_interp} are only $\calO(n\log n)$ and $\calO(n)$, respectively. The article \cite{SK2024} considers another interpolation algorithm based on FFT, with the same computational cost. However, our algorithm achieves a better error decay than the one in \cite{SK2024}.
\end{remark}

\section{Multivariate extension by componentwise transforms}\label{sec:multidim}
This section extends the one-dimensional integration result in Section~\ref{sec:traped} to a multidimensional setting.
We consider a componentwise M\"obius transformation $\phi_{\bsc}:(0,2\pi)^d \to \R^d, $
\[
\phi_{\bsc}(\bstheta)
\coloneqq
\big(c_1 \cot(\theta_1/2),c_2 \cot(\theta_2/2),\ldots, c_d \cot(\theta_d/2) \big),
\]
and aim to approximate the multi-dimensional weighted integral
\begin{align}\label{eq:quadrature-multi}
    I_{\rho}(f) \coloneqq \int_{\R^d} f(\bsx)\rho(\bsx)\rd \bsx
    =
    \int_{(0, 2\pi)^d}f(\phi_{\bsc}(\bstheta))\rho(\phi_{\bsc}(\bstheta)) \, \prod_{k=1}^d c_k\phi'(\theta_k)  \rd \bstheta.
\end{align}
Assuming that the target integrand function $f$ lives in a tensor product of one-dimensional weighted Sobolev spaces 
$W^{\alpha,2}_{\rho,\otimes}(\R^d)$ defined by \eqref{eq:tensor_product}, the transformed integrand in \eqref{eq:quadrature-multi} is in the periodic Sobolev space of the same smoothness $W^{\alpha,2}_{\otimes}(\T^d)$. This is a direct consequence of Lemma~\ref{thm:f-rho-norm}.

\begin{proposition}
    Let $\alpha \in \N$, $\rho(\bsx)=\prod_{k=1}^d \rho_k(x_k)$ with $\rho_k\in\smon$, $f\in W^{\alpha,2}_{\rho,\otimes }(\R^d)$, and 
    \[
    g(\bstheta)\coloneqq f(\phi_{\bsc}(\bstheta))\rho(\phi_{\bsc}(\bstheta)) \prod_{k=1}^d c_k\phi'(\theta_k) .
    \] 
    Then $g\in W^{\alpha,2}_{\otimes} (\T^d)$.
\end{proposition}

As the modified integrand function in \eqref{eq:quadrature-multi} is in the periodic Sobolev space $W^{\alpha,2}_{\otimes}(\T^d)$ that is norm-equivalent to the Sobolev space of dominating mixed smoothness $W^{\alpha,2}_{\rm mix}(\T^d)$, one can employ in \eqref{eq:quadrature-multi} \emph{good rank-$1$ lattice points}~\cite[Equation~(5.27)]{N1992book}, \cite[Section~4.5]{SJ1994book} to obtain the error convergence rate $n^{-\alpha}(\log n)^{\alpha d}$ or alternatively resort to \emph{higher-order digital nets}~\cite{GSY2017} to achieve the exactly optimal rate of $n^{-\alpha}(\log n)^{(d-1)/2}$.

\section{Concluding remarks}\label{sec:conc}

In this paper, we introduced the M\"obius-transformed trapezoidal rule for numerical integration over the real line. We proved that this rule attains the optimal rate of convergence for the worst-case error for a wide class of weighted Sobolev spaces. Let us review some notable features of our method. The assumption $\rho\in\smon$ is general enough to include weights that decay at the speed $\re^{-|x|}$ or even slower as $x$ approaches infinity. Indeed, we can see the expected convergence for an integral weighted by the logistic probability density in Figure~\ref{fig:int-test2}. Moreover, the implementation of our method is straightforward: no information on the smoothness of the inputted integrand function is required, the only needed information on the weight are its values at M\"obius-transformed equidistant points on the unit circle, and the computational cost of the method is low.

As noted already in Section~\ref{sec:intro}, quadrature rules based on a variable transformation and a subsequent application of the trapezoidal rule have a long history~\cite{Sag1964, Schwartz1969, Stenger1973, TakahasiMori1973}.
However, in contrast to the Möbius-transformed trapezoidal rule, most existing rules have been designed to integrate analytic functions.
We highlight the relatively popular \emph{single} and \emph{double exponential formulas}~\cite{TakahasiMori1974,S1997,TO2023} that use the change of variables and what is called a single or double exponential transformation $\psi \colon \R \to I$ to approximate an integral over an interval $I$ as
\begin{equation} \label{eq:DE-formula}
  \int_I f(x) \rd x = \int_\R f(\psi(t)) \psi'(t) \rd t \approx h \sum_{j=-n}^n f(\psi(j h)) \psi'(jh) ,
\end{equation}
where $h > 0$.
These formulas are known to exhibit fast rates of convergence for functions analytic in certain regions of the complex plane~\cite{TanakaSugihara2009}. Even the optimality of the double exponential formula is shown in \cite{S1997} for analytic functions. We also refer to \cite{GKT2024} for recent further results for analytic functions.

When $I = (-1,1)$, the single and double exponential transformations are
\begin{align*}
  \psi_\textup{SE}(x) = \tanh\bigg(\frac{x}{2}\bigg) \quad \text{ and } \quad \psi_\textup{DE}(x) = \tanh\bigg(\frac{\pi}{2} \sinh(x) \bigg) .
\end{align*}
To obtain other quadrature rules such as~\eqref{eq:quadrature} to integrate functions in weighted Sobolev spaces over $\R$, one could replace the Möbius transformation with the inverse of a single or double exponential transformation.
We have observed numerically that the inverse single exponential transformation works well. However, it is not known yet if these formulas can achieve the optimal rate of convergence for weighted Sobolev spaces, and answering this question is left for future studies.

\section*{Acknowledgment}
We thank Yoshihito Kazashi for his valuable comments.

\appendix

\section{Induction proofs}
\label{ap:induction}
The purpose of this appendix is to provide induction proofs for four technical results used in the above analysis, namely the representations \eqref{eq:comb_terms_sth}, \eqref{eq:ind1}, \eqref{eq:ind2} and \eqref{eq:ind1p}.

Let us first prove \eqref{eq:comb_terms_sth}, i.e., that for any $\beta \in \N$, the derivative $(\omega^s)^{(\beta)}$ is a finite linear combination of terms of the form
\begin{equation}
\label{eq:induction1}
\omega^s \prod_{j=1}^{\gamma} \frac{\omega^{(\tau_j)}}{\omega} , \qquad \text{with} \ \ \gamma \leq \beta \ \ \text{and} \ \ \sum_{j=1}^\gamma \tau_j = \beta.
\end{equation}
First of all, for $\beta = 1$,
$$
(\omega^s)' = s \omega^{s} \frac{\omega'}{\omega},
$$
which is of the required form with $\gamma =1$ and $\tau_1 = 1$. Assume then that the claim holds for an arbitrary but fixed $\beta \in \N$. The proof is completed by showing that the derivative of a term of the form \eqref{eq:induction1} is a linear combination of terms that satisfy the same conditions with $\beta$ replaced by $\beta +1$:
\begin{align*}
\bigg( \omega^s \prod_{j=1}^{\gamma} \frac{\omega^{(\tau_j)}}{\omega}\bigg)' &= s  \omega^s \bigg(\frac{\omega'}{\omega} \prod_{j=1}^{\gamma} \frac{\omega^{(\tau_j)}}{\omega} \bigg)+ 
\sum_{k=1}^{\gamma} \omega^s \prod_{j=1}^{\gamma} \frac{\omega^{(\tau_j + \delta_{jk})}}{\omega} 
- 
\gamma  \omega^s  \bigg( \frac{\omega'}{\omega} \prod_{j=1}^{\gamma} \frac{\omega^{(\tau_j)}}{\omega}\bigg) \\
& = (s - \gamma) \omega^s \bigg(\frac{\omega'}{\omega} \prod_{j=1}^{\gamma} \frac{\omega^{(\tau_j)}}{\omega} \bigg) + \sum_{k=1}^{\gamma} \omega^s \prod_{j=1}^{\gamma} \frac{\omega^{(\tau_j + \delta_{jk})}}{\omega}, 
\end{align*}
where $\delta_{jk}$ is the Kronecker delta. As all summands on the right-hand side are of the required form, with the first one having $\gamma +1 \leq \beta + 1$ terms in its product and the others $\gamma$ terms, the assertion follows.

We then prove \eqref{eq:ind1}, i.e., that for $g = ((f \rho) \circ \phi) \phi'$ the derivative $g^{(\tau)}$, $\tau \in \N_0$, is a finite linear combination of terms of the form
\begin{equation}
\label{eq:induction2}
    \big((f^{(\tau_1)} \rho^{(\tau_2)}) \circ \phi \big) \prod_{j=1}^{\tau_1 + \tau_2+1} \phi^{(\tau_{3,j})},
    \quad \text{with} \ \
    \tau_1 + \tau_2 \leq \tau \ \ \text{and}  \  \sum_{j=1}^{\tau_1 + \tau_2 +1} \tau_{3,j} = \tau+1.
\end{equation}
The case $\tau = 0$ obviously holds with $\tau_1, \tau_2 = 0$ and $\tau_{3,1} = 1$. Assume then that the claim is true for an arbitrary but fixed $\tau \in \N_0$. The proof is completed by showing that the derivative of a term of the form \eqref{eq:induction2} is a linear combination of terms that satisfy the same conditions with $\tau$ replaced by $\tau +1$: 
\begin{align*}
\bigg( \big((f^{(\tau_1)} \rho^{(\tau_2)}) \circ \phi \big) \prod_{j=1}^{\tau_1 + \tau_2+1} & \phi^{(\tau_{3,j})}  \bigg)' 
=  \sum_{k=1}^{\tau_1 + \tau_2 +1} \big((f^{(\tau_1)} \rho^{(\tau_2)}) \circ \phi \big) \prod_{j=1}^{\tau_1 + \tau_2+1} \phi^{(\tau_{3,j}+ \delta_{jk})} \\
&+ \big((f^{(\tau_1 +1)} \rho^{(\tau_2)} + f^{(\tau_1)} \rho^{(\tau_2 + 1)})\circ \phi \big) \bigg(\phi' \prod_{j=1}^{\tau_1 + \tau_2+1} \phi^{(\tau_{3,j})} \bigg).
\end{align*}
As all summands on the right-hand side are of the required form, with either $\tau_1$ or $\tau_2$ increasing by one on the second line and neither of the two increasing on the first line when increasing the order of the derivative from $\tau$ to $\tau +1$, the assertion follows.

Then it is the turn of \eqref{eq:ind1p}, i.e., we aim to prove that for $g_p = ((f\,\rho^{1/p})\circ \phi) (\phi')^{1/p}$ the derivative $g_p^{(\tau)}$, $\tau \in \N$, is a finite linear combination of terms of the form
\begin{equation}
\label{eq:induction3}
    \big((f^{(\tau_1)} (\rho^{1/p})^{(\tau_2)}) \circ \phi \big) (\phi')^{1/p - \tau_3} \prod_{j=1}^{\tau_1+\tau_2 + \tau_3} \phi^{(\tau_{4,j})},
    \end{equation}
with
\begin{equation}
\label{eq:induction3_2}
\tau_1 + \tau_2 + \tau_3 \leq \tau \quad \text{and} \quad  \sum_{j=1}^{\tau_1+\tau_2 + \tau_3} \tau_{4,j} = \tau + \tau_3.
    \end{equation}
To prove the case $\tau = 1$, write
\begin{align*}
\Big(\big((f\,\rho^{1/p})\circ \phi\big) (\phi')^{1/p} \Big)' & = \big( (f' \rho^{1/p} + f (\rho^{1/p})') \circ \phi \big) (\phi')^{1/p} \phi' \\ 
& \qquad + \frac{1}{p} \big((f\,\rho^{1/p})\circ \phi\big)  (\phi')^{1/p - 1} \phi^{(2)},
\end{align*}
where all three terms are of the required form with $(\tau_1, \tau_2, \tau_3, \tau_{4,1}) = (1, 0, 0, 1)$, $(0,1, 0, 1)$ and $(0, 0, 1, 2)$, respectively. Assume then that the claim is true for an arbitrary but fixed $\tau \in \N$. The proof is completed by showing that the derivative of a term of the form \eqref{eq:induction3} is a linear combination of terms that satisfy the same conditions \eqref{eq:induction3}-\eqref{eq:induction3_2} with $\tau$ replaced by $\tau +1$: 
\begin{align*}
\bigg( \big(  ( & f^{(\tau_1)} (\rho^{1/p})^{(\tau_2)}) \circ \phi \big) (\phi')^{1/p - \tau_3} \prod_{j=1}^{\tau_1+\tau_2 + \tau_3}  \phi^{(\tau_{4,j})} \bigg)' \\ &= 
\big( \big(f^{(\tau_1+1)} (\rho^{1/p})^{(\tau_2)} + f^{(\tau_1)} (\rho^{1/p})^{(\tau_2+1)} \big) \circ \phi \big) (\phi')^{1/p-\tau_3} \bigg( \phi' \prod_{j=1}^{\tau_1+\tau_2 + \tau_3}  \phi^{(\tau_{4,j})} \bigg) \\
& \qquad + (1/p - \tau_3) \big((f^{(\tau_1)}(\rho^{1/p})^{(\tau_2)}) \circ \phi \big) (\phi')^{1/p - (\tau_3 +1)} \bigg( \phi^{(2)} \prod_{j=1}^{\tau_1+\tau_2 + \tau_3}  \phi^{(\tau_{4,j})} \bigg) \\
& \qquad + \sum_{k=1}^{\tau_1+\tau_2 + \tau_3} \big((f^{(\tau_1)} (\rho^{1/p})^{(\tau_2)}) \circ \phi \big) (\phi')^{1/p - \tau_3} \prod_{j=1}^{\tau_1+\tau_2 + \tau_3} \phi^{(\tau_{4,j} + \delta_{jk})}.
\end{align*}
As all summands on the right-hand side are of the required form, with either $\tau_1$ or $\tau_2$ increasing by one on the first line, $\tau_3$ increasing by one on the second line and none of the three indices increasing on the final line  when increasing the order of the derivative from $\tau$ to $\tau +1$, the assertion follows.

This appendix is completed by proving \eqref{eq:ind2}, i.e., that the derivatives of $\phi (\theta) = -c \cot(\theta/2)$ satisfy
\begin{equation}
\label{eq:induction4}
\phi^{(\tau)}(\theta) = \frac{\psi_\tau(\theta)}{\sin^{\tau+1}(\theta/2)}, \qquad \tau \in \N_0,
\end{equation}
with $\psi_\tau(\theta) \in C^\infty(\R)$ being a bounded finite linear combination of products of trigonometric functions. By definition, the claim holds for $\tau = 0$. Assume that $\phi^{(\tau)}$ is of the form \eqref{eq:induction4} for an arbitrary but fixed $\tau \in \N_0$ and let us differentiate:
\begin{align*}
\phi^{(\tau+1)}(\theta) =  \frac{\psi_\tau'(\theta) \sin(\theta/2) - \tfrac{\tau+1}{2} \psi_\tau(\theta) \cos(\theta/2)}{\sin^{\tau+2}(\theta/2)},
\end{align*}
which immediately proves the claim.

\bibliographystyle{plain}
\bibliography{ref.bib}

\end{document}